\documentclass[11 pt,a4paper]{amsart}
\usepackage[english]{babel}
\usepackage{amsmath}
\usepackage{amstext}
\usepackage{amssymb,epsfig}
\usepackage{amsthm}
\usepackage{amsfonts}
\usepackage{graphicx}
\newenvironment{proofof}[1]{\noindent {\em Proof of #1.}}{ \hfill\qed\\}                                                     

\renewcommand{\footnote}{\endnote}
\newtheorem{theorem}{Theorem}
\newtheorem{lemma}[theorem]{Lemma}
\newtheorem{proposition}[theorem]{Proposition}
\newtheorem{definition}[theorem]{Definition}
\newtheorem{remark}[theorem]{Remark}
\newtheorem{corollary}[theorem]{Corollary}           

\author{Nicolas Bedaride}
\address{LATP  UMR 7353, Universit\'e Paul C\'ezanne, 
avenue escadrille Normandie Niemen, 
13397 Marseille cedex 20, France.
{\it E-mail address:} {\tt nicolas.bedaride@univ-amu.fr}
}

\title{Billiard complexity in rational polyhedra}
\begin{document}
\begin{abstract}
We give a new proof for the directional billiard complexity in the cube, 
which was conjectured in \cite{Ra} and proven in \cite{Ar.Ma.Sh.Ta}. 
Our technique gives us a similar theorem for some rational polyhedra.
\end{abstract}

\maketitle 
\bibliographystyle{plain}
\section{Introduction}
A billiard ball, i.e.\ a point mass, moves inside a polyhedron $P$ with unit 
speed along a straight line until it reaches the boundary $\partial{P}$, then 
instantaneously changes direction according to the mirror law, and continues 
along the new line. 

Label the sides of $P$ by symbols from a finite alphabet $\Sigma$ whose 
cardinality equal the number of faces of $P$. Fix a direction $\omega$ and 
code the orbit of a point by the sequence of sides it hits. Consider the 
set $\mathcal{L}(n,\omega)$ of all words of length $n$  which arise via 
this coding, and let $p(n,\omega)=card(\mathcal{L}(n,\omega))$. This is the 
complexity function in the direction $\omega$, it does not depend on the 
initial point if the billiard map is minimal. We want to compute the 
complexity.

There are no result on the complexity for general polyhedron $P$, the only 
known case is the cube if we use a coding with three letters using the same 
letter for parallel faces. It was first found 
by Arnoux, Mauduit, Shiokawa, and Tamura \cite{Ar.Ma.Sh.Ta, Ar.Ma.Sh.Ta1}. 
They use the fact that the billiard map in the cube is a rotation on 
the torus $\mathbb{T}^{2}$ and compute the numbers of cells of the $n$ 
iterate of the rotation. Their result generalizes the computation of 
the directional complexity in the square where the obtained sequences 
are Sturmian. This result was generalized to a cube in 
$\mathbb{R}^{s}, 2\leq s$ 
by Baryshnikov \cite{Ba}. For the rational polygons 
Hubert has given an exact formula for the directional complexity, 
it is a linear polynomial in $n$ and it does not depend on the 
direction \cite{Hu}.

To give a new proof of the complexity in the cube we consider the notion of 
a generalized diagonal, that is an orbit segment in the fixed direction which 
starts and ends in an edge of the cube. The combinatorial length of a 
generalized diagonal is the number of links. We note $N(\omega,n)$ the 
cardinal of the set of generalized diagonals of lenght $n$. It was considered 
in \cite{Ca.Hu.Tr} for the global complexity of a polygonal billiard. 
The notion of generalized diagonals can be defined for any map with a 
partition, it is enough to exchange the words edge and discontinuity of 
the partition, in the definition.

In the case of the rational polygons the billiard map, in one direction, 
is an interval exchange map, for the polyhedra there is a similar result 
for which we need some definitions.
 We call a polyhedron rational if the group $G$ generated by the linear 
reflections on the faces is finite. Furthermore we call a map an affine 
polygon exchange if there is an finite partition of a polygon $X$  in 
polygons, a map of $X$ on itself which is defined on each polygon by an 
isometry and such that the image of the partition is a partition.

Let $P$ be a rational polyhedron and $G$ his group of reflections. The 
phase space of the biliard map in $P$ has a decomposition into subspaces 
$\partial{P}\times G\omega$. This subspace can be viewed as a polyhedron 
$Q$ where we identify the parallel copies of the same face. The billiard
map in $P$ becomes a geodesic flow in $Q$, and the geodesic flow in a 
given direction yields an affine polygonal 
exchange which each partition element is coded by a single letter.

In general each face is represented several 
times in the polygonal exchange, thus the complexity of the polygonal
exchange yields bounds of the complexity of the billiard.
 However in the case of the cube, when coding the parallel faces by 
the same letters, the symbolic sequence  of a given point produced by 
the billiard is the same as the one produced by the geodesic flow of the 
given point in the cube with parallel faces identified.
 Thus the complexity of the billiard in the cube and its associated 
polygonal exchange are the same. 
 

Our main result relates the complexity of an IDOC2 polygon exchange to the 
number of its generalized diagonals.

\begin{definition}
We consider a affine polygon exchange, the discontinuities are intervals, we 
consider two of them $a\neq  b$. If $T^{n}a\cap T^{m}b$ is a point or is empty and  if $T^{n}\partial{a}\cap T^{m}\partial{b}$ is empty for every $n,m\in 
\mathbb{Z}$ then the polygonal exchange is called IDOC2. 
\end{definition}
\begin{proposition}\label{comp1}
Consider an affine polygon exchange which is IDOC2 and minimal.
We call $p(n)$ the complexity of the polygon exchange, with the natural 
coding and $N(n)$ the number of generalized diagonals of combinatorial 
length $n$ for this map.
$$p(n)=(2-n)p(1)+(n-1)p(2)+\sum_{i=2}^{n-1}{\sum_{j=1}^{i-1}{N(j)}}\quad 
\forall n>2.$$
\end{proposition}
In fact we can apply Proposition \ref{comp1} to some rational polyhedra. 
The simplest ones are the right prisms with polygonal basis. In this case if 
we are able to bound the number of generalized diagonals in one 
direction we obtain bounds for the complexity.
\begin{definition}
Let $P$ be a rational polyhedron and $\omega$ a direction. The 
direction $\omega$ is called $BP$ irrational if there is no diagonal 
in this direction which passes through three sorts of edges and
there is no diagonal which passes through two parallel edges. 
\end{definition}

\begin{corollary}\label{cor}
Let $P$ be a rational polyhedron, $\omega$ a minimal and $BP$ irrational 
direction, and $N(n,\omega)$ the number of generalized diagonals of 
length $n$ in the direction $\omega$. There exists $C\in \mathbb{N}$ 
and $b>0$  constants such that for all $n>2$.
$$p(n,\omega)\leq (2-n)p(1,\omega)+(n-1)p(2,\omega)+b\sum_{i=2}^{n-1}{\sum_{j=1}^{i-1}{N(j,\omega)}}\leq p(C+n,\omega).$$
\end{corollary}
To apply Proposition \ref{comp1} to the cube we have to compute the number of 
generalized diagonals and to understand the IDOC2 condition in the case
of the cube.
\begin{definition}
The initial direction of the billiard map is a vector $w=(w_{1},w_{2},w_{3})
\in \mathbb{R}^{3}$. If the $\omega_{i}$ are rationally independant we say 
that the direction is totally irrational, and if we add the condition that the
 $(\omega_{i}^{-1})$ \text{are} $\mathbb{Q}$ independant we call $\omega$ $B$ irrational. 
\end{definition}
The condition of total irrationality is the translation to the cube of the $BP$ condition, and it implies that the polygon exchange is IDOC2.
\begin{theorem}\label{comp}
Let $\omega$ be a $B$ irrational direction, then the directional billiard 
complexity in the cube satisfies
$$p(n,\omega)=n^{2}+n+1 \quad \forall n>0.$$
\end{theorem}
\begin{remark}
In \cite{Ar.Ma.Sh.Ta,Ar.Ma.Sh.Ta1} the theorem was proven with 
the condition of total irrationality, but there exists some directions $\omega$ which are totally irrational, not $B$ irrational and such that there exists
 $n$ with $p(n,\omega)<n^{2}+n+1$. The mistake in their proof is minor. In 
\cite{Ba} an alternative proof was given with the condition of $B$ 
irrationality.

Proposition \ref{comp1} is true under a condition weaker than the affinity 
assumed. Without modification of its proof, the relation 
between the global complexity and the number of generalized 
diagonals of combinatoric lenght $n$ holds for polygonal billiards 
generalizing the result of \cite{Ca.Hu.Tr} to the non convex case.
\end{remark}  


We obtain a similar result for some right prisms.
\begin{theorem}\label{comp2}

Let a right prism with a tiling polygon as base. 
Consider a natural coding on this billiard table and 
let $\omega$ be a minimal and $BP$ irrational direction, then there
exists positives constants $A,B$ such that
$$B\leq \frac{p(n,\omega)}{n^{2}}\leq A \quad \forall n>0.$$

\end{theorem} 
\section{Polygonal exchanges}
\begin{proofof}{Proposition \ref{comp1}}
First of all we have to recall some results of words combinatorics \cite{Ca}.

For any $n\geq1$ let $s(n):=p(n+1)-p(n)$. For $v \in \mathcal{L}(n)$ let 

$$m_{l}(v)= card\{a\in \Sigma,va\in \mathcal{L}(n+1)\},$$
$$m_{r}(v)= card\{b\in \Sigma,bv\in \mathcal{L}(n+1)\},$$
$$m_{b}(v)= card\{a\in \Sigma, b\in \Sigma,bva\in \mathcal{L}(n+2)\}.$$

A word is call right special if $m_{r}(v)\geq 2$, left special if 
$m_{l}(v)\geq 2$ and bispecial if it is right and left special. 
Let $\mathcal{BL}(n)$ be the set of the bispecial words. Cassaigne \cite{Ca} has shown:
\begin{lemma}\label{jul}
$$s(n+1)-s(n)=\sum_{v\in \mathcal{BL}(n)}{m_{b}(v)-m_{r}(v)-m_{l}(v)+1}.$$
\end{lemma}

For the proof of the lemma we refer to \cite{Ca} or \cite{Ca.Hu.Tr}. 

Since the map is minimal we just have to count the number of initial 
words of length $n$.
\begin{figure}[hbt]
\begin{center}
\includegraphics[width=6 cm]{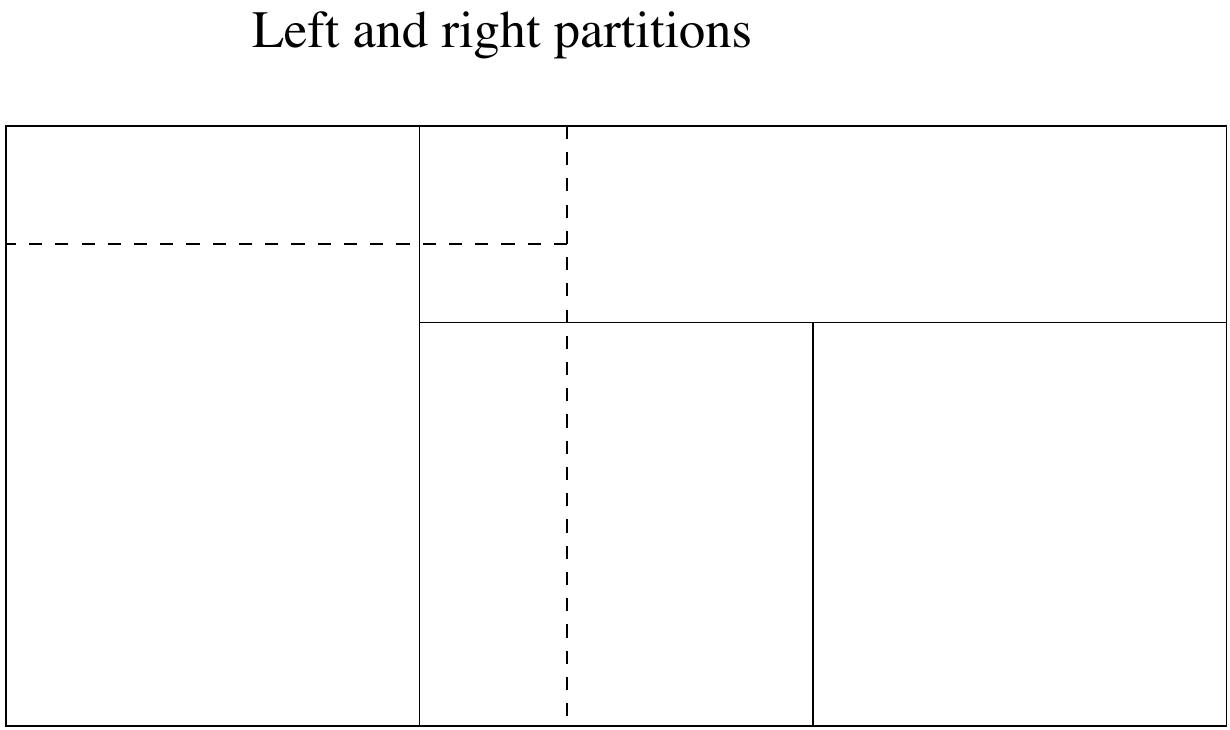}
\includegraphics[ width=7cm]{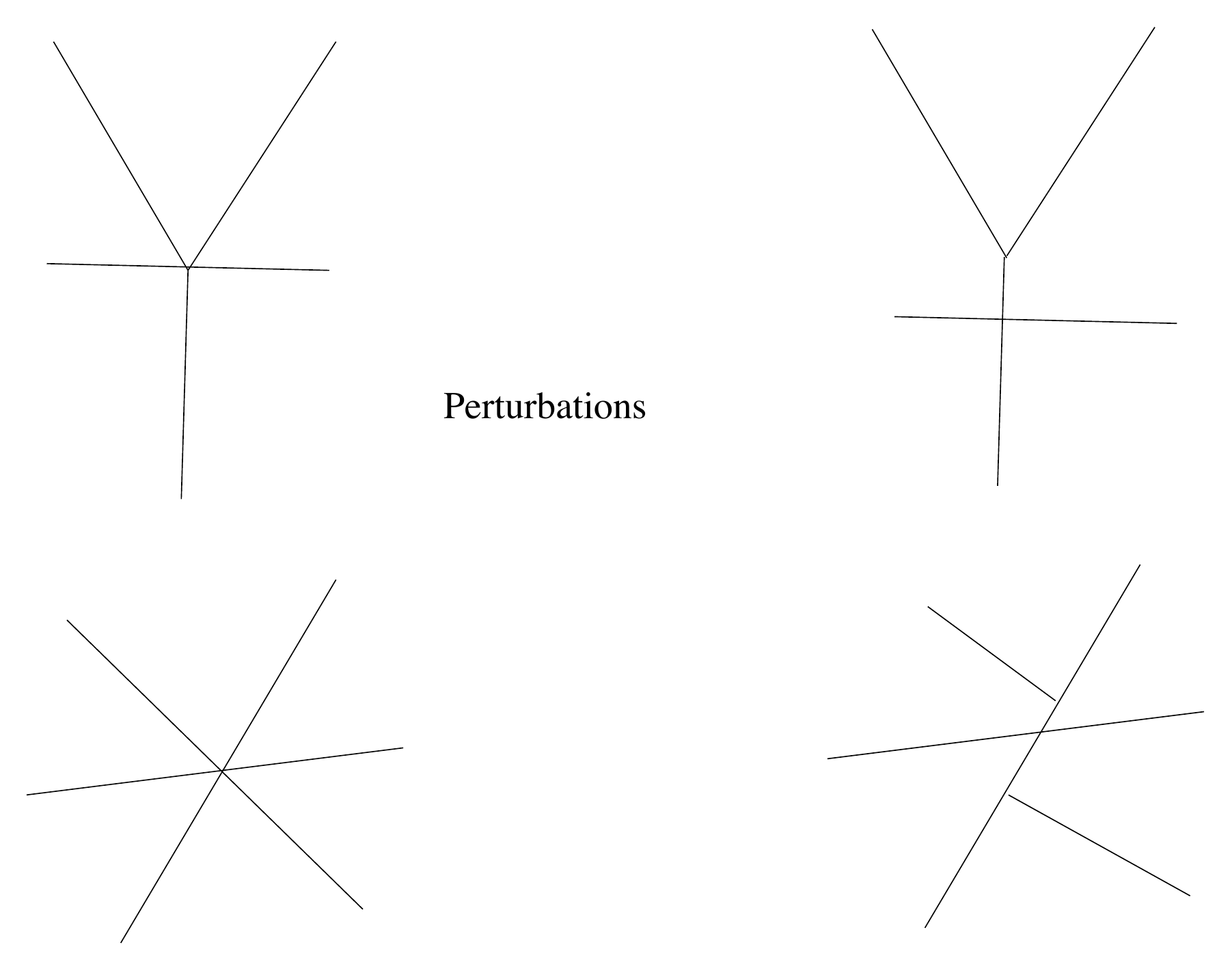}
\caption{\label{fig1}}
\end{center}
\end{figure}
We consider an affine polygon exchange $T$, we label the polygons of the 
partition $\mathcal{P}$ with the letters from a finite alphabet. Let a 
polygon of the partition corresponding to a letter $a$, it has a image by 
$T$ which is an union of polygons. These polygons are associated to the 
letters $b$ such that $ab$ is a word of the language. So a word $v$ of 
the langague is represented by a polygon, it is partitionned into 
$m_{r}(v)$ polygons each of which corresponds to a different word. Since 
$T$ is invertible we can repeat the same argument with $T^{-1}$ and 
$m_{l}(v)$. If we intersect the two partitions of the polygon associated 
to $v$ we obtain a partition into $m_{b}(v)$ polygons.

Now we use the Euler formula $F=E-V+1$ for a partition of a polygon 
into $F$ polygons with $E$ edges and $V$ verteces. There are three 
partitions so we obtain three equations.

$$m_{r}(v)=E_{r}(v)-V_{r}(v)+1,$$
$$m_{l}(v)=E_{l}(v)-V_{l}(v)+1,$$
$$m_{b}(v)=E_{b}(v)-V_{b}(v)+1.$$

The verteces of the third partition are the verteces of the first one plus 
the verteces of the second one and the verteces created by the intersection of an edge of the first partition and an edge of the second one. Let us call this number $k(v)$, by the IDOC2 we obtain $V_{b}(v)-V_{r}(v)-V_{l}(v)=k(v)$.

For the edges we use the IDOC2 condition, and we assume, for the moment that 
an image of an edge does not intersect a vertex. The IDOC2 tells us that 
two edges can only intersect in one point, and we see that each point $k(v)$
 create two edges, so $E_{b}(v)-E_{r}(v)-E_{l}(v)=2k(v)$, thus $m_{b}(v)-m_{r}(v)-m_{l}(v)+1=k(v)$.

If we are in the case where an image of an edge intersect a vertex then we 
can make a small perturbation of some edge of one partition, see Figure 
\ref{fig1}, 
such that the number of polygons does not change, the formula is now true and 
we obtain $$m_{b}(v)-m_{r}(v)-m_{l}(v)+1=k(v).$$   
To finish the proof we just have to remark that $\sum_{v\in \mathcal{BL}(n)}{k(v)}=N(n).$ 
\end{proofof}
\section{Generalized diagonals for the cubic billiard}
 \begin{proofof}{Theorem \ref{comp}}
 When we play billiard in a polyhedron  we can reflect the line or we can 
reflect the polyhedron and follow the same line, it is the unfolding.
\begin{figure}[hbt]
\begin{center}
\includegraphics[width= 6 cm]{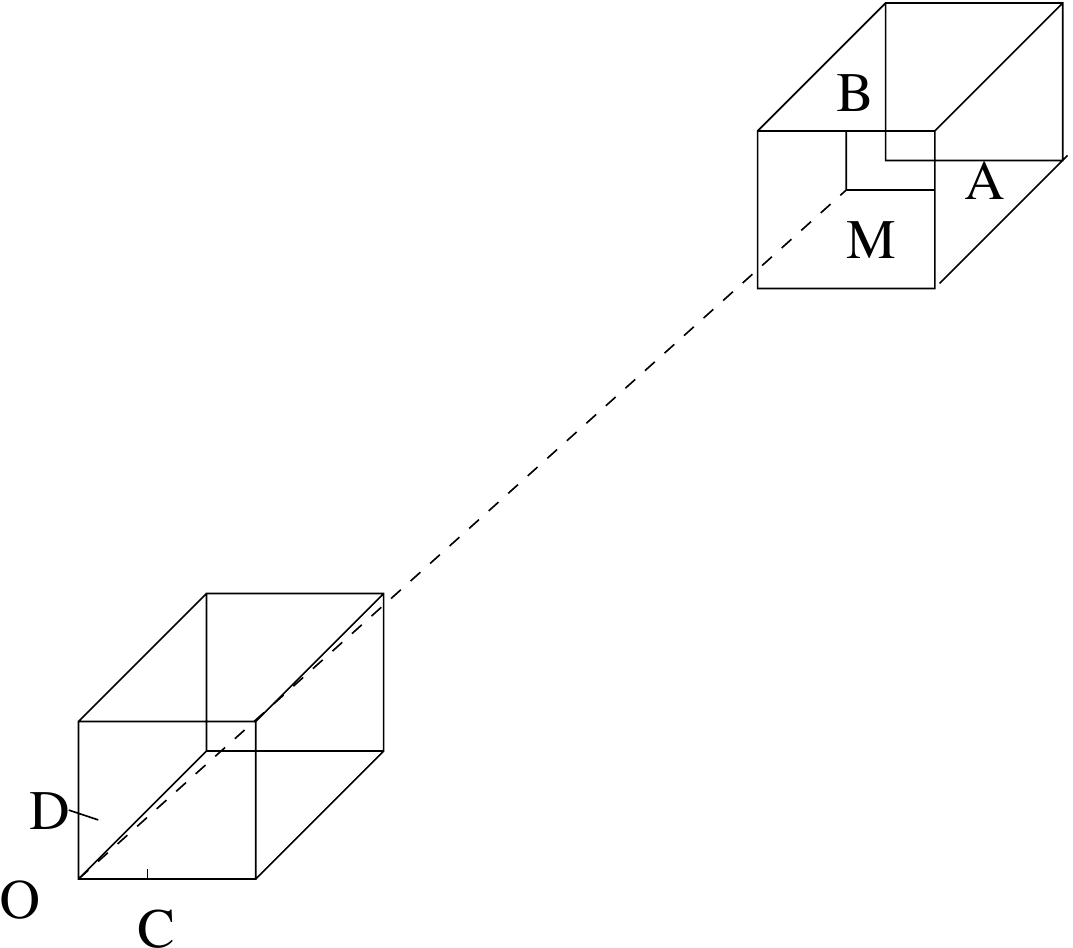}
\caption{\label{fig2}}
\end{center}
\end{figure} 
Let $O$ be a vertex of the cube and consider the segment of direction 
$\omega$ who start from $O$ and end at a point $M$ after it pass through 
$n$ cubes. $M$ is a point of a face of an unfolding cube, if we translate 
$M$ with a direction parallel to one of the two directions of the face we 
obtain a point $A$ on an edge and if we call $C$ the point such that 
$\vec{OC}=\vec{MA}$\quad then $CA$ is a generalized diagonal, and we have another 
one, $DB$ in the figure, arising from the second translation.

The symmetries of the cube implies that these diagonals are the only one.
It remains to prove that the two generalized diagonals are of combinatorial 
length $n$.  

The first thing to remark is that the condition of total irrationality 
implies that a generalized diagonal can not begin and end on two parallel 
edges. So the edges of begin and end are of different type.


To see that the combinatorial length is equal to $n$ we can remark that the 
sum of the length of the projections is twice the length of the trajectory, 
so we just have to prove it for the projection, i.e.\ billiard in the 
square, where it follows from the symmetry. Lemma \ref{jul} shows us that the 
diagonals are non degenerate.

So we obtain that $s(n+1)-s(n)=2$, since $p(1)=3 \quad p(2)=7$, we obtain $p(n,\omega)=n^{2}+n+1$.
\end{proofof}

Next we show that the $B$ condition of irrationality implies the 
impossibility to have a trajectory which pass through all three types 
of edges. It is the point that was wrong in the articles 
\cite{Ar.Ma.Sh.Ta,Ar.Ma.Sh.Ta1}.
\begin{lemma}
There exist a total irrational direction such that there exists $n$ and 
a trajectory of length $n$ which pass through three different edges.

If $\omega$ is a $B$ irrational direction  this is impossible.
\end{lemma}
\begin{proof}
For the first point just consider the direction 
$(1-\frac{\pi}{6},1,\frac{6-\pi}{12-\pi})$ and the points on the edges 
$(\pi/6,0,0)$,$(1,1,\frac{6-\pi}{12-\pi})$,$(2,\frac{12-\pi}{6-\pi},1)$.

The proof of the second point is  by contradiction.

Let $\omega$  be a $B$ irrational direction and three points on some 
edges $(x,0,0)$,\quad $(a,y,b)$,\quad $(c,d,z)$ with $a,b,c,d \in \mathbb{N}$. 

The fact that the points are on a line with direction $\omega$ gives 
$$\frac{c-x}{a-x}=\frac{d}{y}=\frac{z}{b}.$$
Furthermore $\omega$ is a multiple of $(a-x,y,b)$ and of $(c-x,d,z)$, thus
$$\frac{c-x}{\omega_{1}}=\frac{d}{\omega_{2}} \quad \frac{a-x}{\omega_{1}}=\frac{b}{\omega_{3}}.$$
So we have $\frac{a-c}{\omega_{1}}=\frac{b}{\omega_{3}} -\frac{d}{\omega_{2}}$ which is excluded by $B$ irrationality.
\end{proof}
\section{Right prisms.}
We consider a right prism with a tiling polygon for base. To apply
 Corollary \ref{cor}, we need to count the 
generalized diagonals. The same construction as for the cube 
works but there is less 
symmetry so we must consider all the verteces of the base polygon and 
we remark that for one vertex the number of generalized diagonal can 
be null.


\begin{figure}[hbt]
\begin{center}
\includegraphics[width= 8cm]{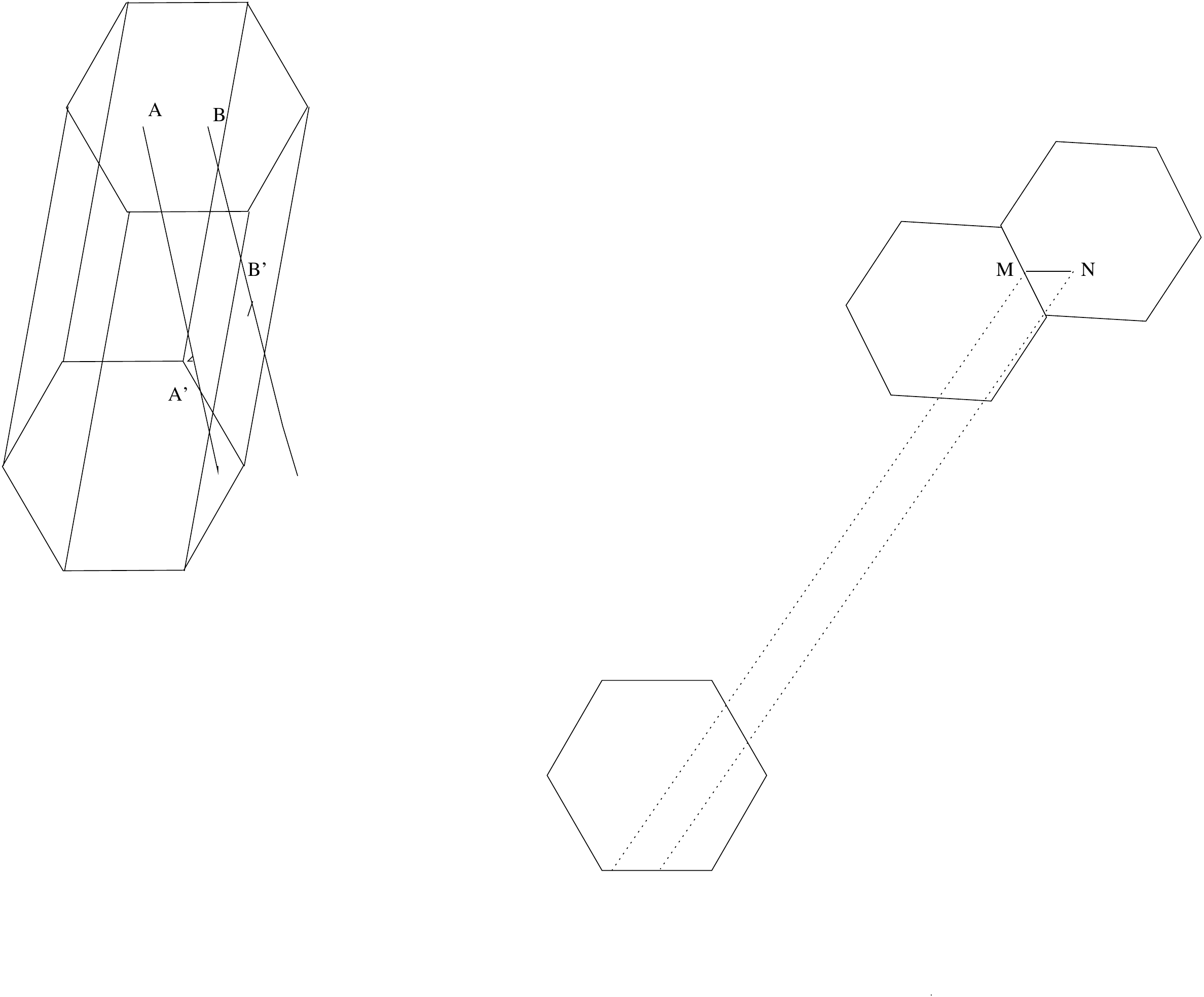}
\caption{\label{fig3}}
\end{center}
\end{figure}
 \begin{lemma}\label{diag1}
Let $P$ be a right prism and $\omega$ a $BP$ irrational direction. 
Consider the two diagonals we have constructed, they are of combinatorial 
length $n$. 
\end{lemma}
\begin{proof}

Consider the following induction hypothesis for a right prism:

In a fixed direction all the trajectories who started from the same edge 
and stop on the same face of an unfolding polyhedron are of the same 
combinatorial length.

For $n=1$ it is obvious, let us assume it is true for $n$ and consider two 
diagonals of the rank $n+1$ who started from the same edge. Call their ending 
points $A$ and $B$ and follow the trajectories in the inverse direction. The 
intersection with the first faces are $A'$, $B'$, see Figure \ref{fig3}, if 
they are on the same 
face we can apply the hypothesis, else there exists a finite sequence of faces who connect these faces, we consider the points on the common 
edges and we apply the same reasoning.
\end{proof}

\begin{lemma}\label{dirdiag2}
Let $P$ be a right prism with regular hexagonal base and $\omega$ a $BP$ 
irrational direction. It is impossible to have two diagonals of the same 
length which start from the same edge.

If $P$ is a right prism with tiling triangle base the number of 
diagonals which start from an edge is bounded.
\end{lemma}
\begin{proof}
The proof is by contradiction.
Suppose the lenght of a side is one.

Assume that the initial edge is on the hexagon.
Suppose that on this edge two diagonals start and are of length $n$.
If the end points of the diagonals are on vertical edges we can 
translate one point to the other by a horizontal translation of lenght 
less than one, contradiction. 
Consider the plane parallel to the horizontal one and who pass on one 
of the end points $M$, see Figure \ref{fig3}. This plane intersects the 
diagonals 
on two points $M$ and $N$. If $M$ and $N$ are in the same hexagon by the 
proof of the preceding lemma we have the same length for the two segments,
 so it is impossible. Since the hexagon tiles 
the plane the polygons do not overlap, thus we must have a point on an 
edge of an hexagon which by translation of length less than one parallel 
to one side go out of the hexagon, contradiction. 

If the initial edge is vertical the final edges are of different hexagons,
 and a vertical translation of lenght less than one 
moves one to the other, contradiction.

For the tiling triangles
Figure \ref{fig4} gives us that the number of diagonals 
starting from an edge is bounded.

\end{proof}
\begin{figure}[hbt]
\begin{center}
\includegraphics[width= 5cm]{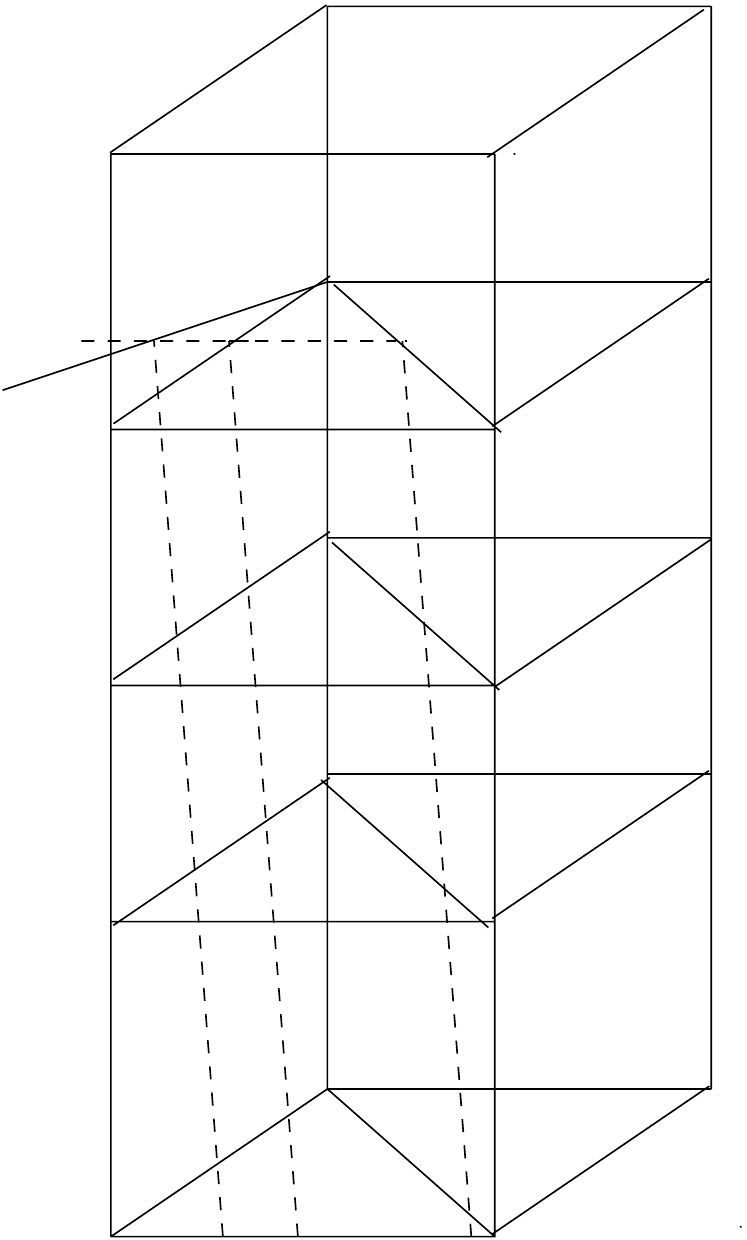}
\caption{\label{fig4}}
\end{center}
\end{figure}
\begin{proofof}{Corollary \ref{cor}}
We introduce another coding. The polyhedron $Q$ is 
made by gluing some copies of $P$. The new coding associates one letter to 
each polygon of the partition of $Q$. The orbit of a point has now 
three codings: one for the billiard map $M$, one for the new coding $M'$ 
and one for the natural coding of the polygonal exchange $M''$ 
whose complexities are denoted by $p(M, n), p(M', n),p(M'', n)$ respectively.
Since there any many copies of $P$ tiling $Q$ we have $p(M, n)\leq p(M', n)$.


The number of polygons of the exchange is related to the
number of polygons of the partition of $Q$. One face gives one 
polygon of the exchange, except if the preimages of singularities
 intersect this face. In this case the face is cut into several polygons,
thus $p(M', n)\leq p(M'', n)$. 
Since the number of singularities is bounded (by a constant $C$) we have 
that a word of length
one in $M''$ is a word of length less than $C$ in $M'$. We deduce
that $p(M'', n)\leq p(M', C+n)$. 
The computation of $p(M'', n)$ can be completed using Proposition \ref{comp1} 
when we 
remark that $N(n,\omega)=bN(n,\omega)$ where $b$ is the number of polygons of the exchange.
\end{proofof}

\begin{proofof}{Theorem \ref{comp2}}
The coding is the natural one, with a letter for each face. The billiard map 
in the direction $\omega$ is a polygon exchange. 
Lemma \ref{diag1} shows that the diagonals are of combinatorial length $n$.
Lemma \ref{dirdiag2} implies that the number of generalized diagonals is bounded, 
the numbers of polygons of the exchange is also bounded, thus
 $bN(n,\omega)$ is bounded.




For a triangle a line which passes through an interior point of the 
triangle and which is parallel to an edge intersect the triangle. For 
a hexagon it is false but if we consider the lines parallel to two edges,
 one of them intersect the hexagon. So there is always at least
one generalized diagonal of length $n$. So there exists $D$ such that 
$1\leq N(n,\omega)\leq D$ and
Corollary 4 implies the existence of $a,d$ with 
$$a\leq p(M'', n)/n^{2}\leq d.$$

Using $p(M', n)\leq p(M'', n)\leq p(M', C+n)$ we obtain the existence 
of $A,B$ such that 
$$B\leq p(M', n)/n^{2}\leq A.$$

Like in the two dimensional case, \cite{Hu}, we prove.
\begin{lemma}
There exists an integer $m$ such that 
$$p(M, n)=p(M', n)\quad  \forall n>m.$$
\end{lemma}
\begin{proof}
The proof is by contradiction.
The language $M'$ is bigger than $M$ and a word of $M'$ gives a word of 
$M$ by a projection, $\pi$, letter by letter.

If it is false we can find for any $n$ two different words $v_{n}$ and 
$w_{n}$ of $M'$ such that
theirs projections in the language $M$ are equal.
We can assume that the words $v_{n}$ are succesive and the same thing for
 $w_{n}$. Let us call $v$ and $w$ the two limit points of the sequences 
$(v_{n})$ and $(w_{n})$.

 There exists $m\in \partial{P}$ and $\phi$ a direction such that
the orbit of $(m,\phi)$ has $\pi(v)$ as coding. The direction is unique,
not necessary the point $m$, see \cite{Ga.Kr.Tr}.

If the orbit of $(m,\phi)$ is singular then there exists 
$e \in \partial{P}$ $\phi'$ a direction, $k$ an integer
 such that the forward orbit of $(e,\phi')$ has $T^{k}(\pi(v))$ as coding 
and is non singular.
 Moreover $T^{-1}(e,\phi)$ belongs to an edge of $P$.

It remains to prove that the point on $\partial{P}$ is unique.

If a point of a face has an orbit which passes through the same face with 
the same direction, the tiling property of $P$ implies that those two faces 
are parallel. So the isometries of the polygonal exchange are translations,
 which implies that it is impossible to have two points with the same coding.  

The point on $\partial{P}$ is unique, and we have $v=w$.
 
\end{proof}

So the last inequality is true for $p(n)$ for $n$ large, if we 
change the constant $B$ it is still true for all the integers.
\end{proofof}
\bibliography{bibliobary.bib} 
\end{document}